\documentclass[12pt]{article}

\usepackage{latexsym,amssymb,amsmath}

\newcommand{\C}{\mathbb C}
\newcommand{\R}{\mathbb R}
\newcommand{\Z}{\mathbb Z}
\newcommand{\Q}{\mathbb Q}
\newcommand{\N}{\mathbb N}
\newcommand{\E}{\mathbb E}
\newcommand{\HH}{\mathbb H}
\newcommand{\F}{\mathbb F}
\newcommand{\slt}{SL(2,\mathbb C)}
\newcommand{\sma}{\left(\begin{array}}
\newcommand{\fma}{\end{array}\right)}

\newtheorem{lem}{Lemma}[section]

\newtheorem{co}[lem]{Corollary}
\newtheorem{thm}[lem]{Theorem}
\newtheorem{prop}[lem]{Proposition}

\newenvironment{proof}{\textbf{Proof.}}{\newline\hspace*{\fill}{$\Box$}\\}

\begin{document}
\title{Groups possessing only indiscrete embeddings in $\slt$}
\author{J.\,O.\,Button\\
Selwyn College\\
University of Cambridge\\
Cambridge CB3 9DQ\\
U.K.\\
\texttt{jb128@dpmms.cam.ac.uk}}
\date{}
\maketitle
\begin{abstract}
We give results on when a finitely generated group has only indiscrete
embeddings in $\slt$, with particular reference to 3-manifold groups. For
instance if we glue two copies of the figure 8 knot along its torus    
boundary then the fundamental group of the resulting closed 3-manifold
sometimes embeds in $\slt$ and sometimes does not,
depending on the identification.
We also give another quick counterexample to Minsky's simple loop question.
\end{abstract}
\section{Introduction}

The group $\slt$ plays an important role in both algebra and geometry.
In particular we obtain on quotienting out by $\{\pm I\}$ the group
$P\slt$ of M\"obius transformations which acts as the group of orientation
preserving isometries on the hyperbolic space $\HH^3$. A discrete subgroup
$\Gamma$ of $P\slt$ is a Kleinian group and these have been much studied 
because the quotient $\Gamma\backslash\HH^3$ 
is a 3-orbifold, and even a 3-manifold
if $\Gamma$ is torsion free,  with a complete hyperbolic metric. However
we can approach this from a group theoretic point of view and ask
what other subgroups of $P\slt$ occur. Certainly one can have abstract
groups which appear both as discrete and indiscrete subgroups of $P\slt$,
such as free groups. However these will not provide new examples of
subgroups of $P\slt$ if we only consider them up to isomorphism so our
focus in this paper will be to look at abstract groups which are subgroups
of $\slt$ but which have no discrete embedding in $\slt$. In this paper
our groups will nearly always be finitely generated and torsion free. In
particular a torsion free subgroup $G$ of $\slt$, or even one without elements
of order 2, will also embed in $P\slt$ because it misses the kernel
$\{\pm I\}$. Moreover if this abstract group $G$ with no elements of order 2
could be realised in some way as a discrete subgroup of $P\slt$ then  
\cite{cul} shows that $G$ lifts to $\slt$, whereupon it will also
be a discrete subgroup of $\slt$. Therefore there is no harm in sticking
to $\slt$ throughout which allows us to work directly with 2 by 2
matrices.

Of course there are many finitely generated, torsion free 
groups which do not embed in $\slt$, either as a
discrete or indiscrete subgroup. Arguably the first big restriction on
subgroups of $\slt$ which comes to mind is that of being commutative
transitive: a group $G$ is commutative transitive (CT) if the relation of
two elements commuting is an equivalence relation on $G-\{e\}$. It is 
straightforward to show
and well known that a torsion free subgroup of $\slt$
(again this generalises to not containing an element of order 2) will be
commutative transitive. Therefore all our examples in this paper must
come from this class. A strictly stronger condition that will also be
of use to us is that of a CSA group, meaning that all centralisers are
malnormal. We introduce basic properties of CT and CSA groups in Section 2,
including a characterisation of the CSA subgroups of $\slt$ in Corollary
2.3.

In section 3 we start by giving standard examples of subgroups of
$\slt$ which do not have a discrete embedding. It goes back to Nisnevi\u{c}
in 1940 that (on avoiding the element $-I$) a free product of subgroups
of $\slt$ also embeds in $\slt$. Therefore we can give examples in Lemma 3.1
of subgroups of $\slt$ which are word hyperbolic and are the fundamental
group of a closed orientable 3-manifold, but which have no discrete
embedding. This result of Nisnevi\u{c} was rediscovered by Shalen in 
\cite{sha} where he also gave conditions under which a free product
with abelian amalgamation of two subgroups of $\slt$ also embeds in
$\slt$. We use and adapt this result to give examples
of subgroups of $\slt$ with no discrete embedding which are
word hyperbolic but which this time cannot be the fundamental group of
any 3-manifold.

It is known that limit groups (finitely generated groups which are fully
residually free) embed in $\slt$ so any limit group which is not the
fundamental group of a 3-manifold (say one that contains $\Z^4$) will
also only have indiscrete embeddings. In Section 4 we show that our
examples are not limit groups using \cite{wlt} which examined which
3-manifolds can have a fundamental group that is a limit group. A
cyclically pinched group is one formed by amalgamating two free groups
over a non identity
element which is not part of a free basis for either free group.
If this element is not a proper power in either factor then the
resulting group embeds in $\slt$. We can therefore ask whether it is always
a limit group. Although it seems to be a folklore result that the answer is
no, we have never seen an example written down. Here we demonstrate that
any cyclically pinched group formed by amalgamating a commutator on one
side with a product of two proper powers of distinct commutators on the other
side is not a limit group if the powers differ by at least 3. The proof is
short and uses standard facts about stable commutator length, including the
lower bound for stable commutator length in a free group.

In Section 5 we look at which closed 3-manifolds, other than orientable
hyperbolic 3-manifolds of course, have fundamental groups which embed in
$\slt$. Although we do not give a full list, we suggest that it is a
somewhat
larger list than might be first thought. In Theorem 5.1 we give the
complete list of torus bundles fibred over the circle with such fundamental
groups: as well as the trivial bundle it is precisely those with Sol
geometry. Given the constructions in Section 3
of groups embedding in $\slt$ using amalgamation over abelian subgroups,
one might suppose that the fundamental group of
a closed orientable 3-manifold admitting a JSJ
decomposition along tori with all pieces hyperbolic might embed. However
we show in Theorem 5.2 that if we glue two copies of the figure 8 knot along
each boundary torus such that the meridians are identified then the  
fundamental group of the resulting 3-manifold does not embed in $\slt$
unless the longitudes are also glued to each other. In this latter case
the group does embed, giving a genus 2 surface bundle fibred over the
circle with fundamental group embedding in $\slt$ even though the
homeomorphism is not pseudo-Anosov. This work utilises a result of
Whittemore from 1973 giving all representations of the figure 8 knot
group in $\slt$.

Having looked at free products with amalgamation, we discuss in Section 6
whether
an HNN extension $\Gamma$ over $\Z$ of a subgroup $G$
of $\slt$ also embeds in $\slt$.
This case seems much less clear and we content ourselves with remarking
that if the stable letter of $\Gamma$
conjugates the generator of $\Z$ to its inverse then the 
answer is no, whereas if $G-\{I\}$ has no elements of trace $\pm 2$ and
the stable letter conjugates the generator to itself then the answer is yes.
However the later case is enough to show quickly and directly that all
limit groups embed in $\slt$ by taking iteratively a free extension of
centralisers.
We also give an example of a group admitting a graph of groups decomposition
with non abelian free vertex groups and maximal cyclic edge groups which
is word hyperbolic but which does not embed in $\slt$. This is in contrast
to when the graph of groups is a tree where we do embed.

Finally we give a concrete counterexample with a brief proof
to Minsky's simple loop question in \cite{min} for genus 2. This is
an example where the fundamental group of the closed orientable surface
of genus 2 has a non injective homomorphism to $\slt$ but such that
the kernel contains no elements represented by simple closed curves. 
We are able to obtain this example quickly by using our results on
embeddings of groups in $\slt$ from Section 3 and the embeddings of the
fundamental group of the figure 8 knot used in Section 5. This can be
seen as one in a sequence of successively shorter counterexamples to
this question, as in \cite{cm}, \cite{ll} and \cite{clsl}.      

\section{Commutative Transitive groups}

We say that
a group $G$ is {\bf commutative transitive} or {\bf CT} for short if the
relation of two elements commuting is transitive on $G-\{e\}$. This is
the same as saying that the centralisers of all (non identity) elements   
are abelian and such groups are also called CA (standing for
centraliser abelian) groups. The finite CT groups are known but there are
many interesting examples of infinite CT groups. For instance Corollary 1
of \cite{ata} states that the free
Burnside groups of sufficiently large odd period are also CT groups.
Here we
are mainly interested in torsion free groups. An application of Zorn's
Lemma tells us that in any group $G$ every element of $G$ is contained in
a maximal abelian group. This can also be used to show that a group $G$ is
CT if and only if every non identity 
element $g$ of $G$ is contained in a unique maximal abelian subgroup, in which
case the centralisers (except $G_G(e)$ if $G$ is non abelian) are precisely
the maximal abelian subgroups. Examples of CT groups are
free groups, limit groups and torsion free word hyperbolic groups.
Moreover it is straightforward to see that any subgroup of $\slt$
not containing $-I$ is CT because in this group a diagonal matrix not equal to 
$\pm I$ can only commute
with a diagonal matrix whereas a matrix of the form $\pm \sma {rr}1&x\\0&1\fma$
can only commute with one of the same form (if $x\neq 0$) 
but every matrix is conjugate in $\slt$ to one of these.

Of course many groups are not CT but here is a quick and easy source of
examples.
\begin{lem}
If $G$ is torsion free and CT but there exists a non identity $x\in G$
whose centraliser $C_G(x)$ has finite index in $G$ then $G$ is abelian.
Thus if $G$ is virtually nilpotent, torsion free and CT then $G$ must
be abelian.
\end{lem}
\begin{proof} For any $g\in G$ there is $n\in \N$ with $e\neq g^n\in C_G(x)$,
so CT implies that $C_G(x)=G$ and hence all elements of $G$ commute with
each other. The last part follows because if $G$ has a finite index subgroup
$H$ which is nilpotent then $H$ has a non trivial centre, thus we have an
element $h\neq e$ with $H=C_H(h)\le C_G(h)$ having finite index in $G$.
\end{proof}

There is a condition that is stronger than being
CT but which is sometimes useful. We say that a subgroup $H$ of a group $G$ 
is {\bf malnormal} (or conjugate separated)
in $G$ if $gHg^{-1}\cap H=\{e\}$ for all $g\in G-H$. A group $G$ is then called
{\bf CSA} (standing for conjugate separated abelian) if every maximal
abelian subgroup is malnormal. This implies CT because we can take a 
non identity element $g\in G$ and a maximal abelian subgroup $M$ containing 
$g$. Then if $x$ is an element of $C_G(g)$ we have $g\in xMx^{-1}\cap M$, 
so $x$ is in $M$ too. In fact an equivalent definition of a CSA group
is that all centralisers are malnormal: if the latter holds then on taking
$x\in C_G(g)$ and any $y$ that commutes with $x$, we obtain $x\in 
yC_G(g)y^{-1}\cap C_G(g)$. Hence malnormality of $C_G(g)$ implies that
$y$ commutes with $g$ so here the centralisers are again 
the maximal abelian subgroups.

Again free groups, limit groups and torsion free
word hyperbolic groups are CSA, and both the CSA and the CT properties
are preserved under taking subgroups, 
but now the situation for subgroups of
$\slt$ not containing $-I$ is a little different.
\begin{lem} If $G$ is a group contained in $\slt$ with $-I\notin G$
and $D$ is the subgroup of diagonal elements in $G$ then $D$ is malnormal in
$G$.
\end{lem}
\begin{proof}
Suppose that
\[g=\sma{cc} \alpha&\beta\\\gamma&\delta\fma\in G
\mbox{ and }d_1=\sma{cl}\lambda_1&0\\0&\lambda_1^{-1}\fma,
d_2=\sma{cl}\lambda_2&0\\0&\lambda_2^{-1}\fma\in D\]
with $gd_1=d_2g$ holding in $G$.
An easy check shows that if  
$\lambda_1\neq\lambda_2$ then $\alpha=\delta=0$, so $g$ is
a matrix of trace 0 and determinant 1 implying that
$g^2=-I\in G$, and if $\lambda_1\lambda_2\neq 1$ then
$\beta=\gamma=0$ so $g\in D$. This only leaves $d_1=d_2=\pm I$.
\end{proof}
\begin{co}
If $G$ is a subgroup of $\slt$ such that the only element of $G$ with trace
in $\{-2,2\}$ is $I$ then $G$ is a CSA group.

More generally a non abelian subgroup $G$ of $\slt$ is CSA if and only if
it does not contain $-I$ and for any non identity element $g$ with 
trace$(g)=\pm 2$ and $\gamma_g\in\slt$ such that $\gamma_gg\gamma_g^{-1}=
\pm\sma{cc}1&x\\0&1\fma$, the conjugate group $\gamma_gG\gamma_g^{-1}$
contains no elements with bottom left hand entry 0 other than those
with trace $\pm 2$. 
\end{co}
\begin{proof}
On being given an element $g$ of $G$ with trace not equal to 
$\pm 2$ we can assume by conjugation that $g$ is a diagonal matrix not 
equal to $\pm I$. Clearly the centraliser $C_G(g)$ is equal to the abelian
subgroup of diagonal elements in $G$ and so is malnormal by Lemma 2.2,
as $-I$ is not in any conjugate of $G$.
This completes the first case.

Now suppose (conjugating if necessary) that $g=\pm\sma{cc}1&x\\0&1\fma$
for $x\neq 0$. Then $C_G(g)$ 
is also an abelian subgroup, consisting of all matrices in $G$ with trace
$\pm 2$ and bottom left hand entry 0. However an element of $\slt$ can
only conjugate one (non identity) element of this form into another if its
bottom left hand entry is 0 as well. If this element is in $G$ then it lies
in $C_G(g)$ if and only if its trace is $\pm 2$, by the given condition.    
\end{proof}

For $m\neq 0$ the Baumslag-Solitar group $BS(1,m)$ is given by the
presentation $\langle x,t|txt^{-1}=x^m\rangle$. These groups embed
(always indiscretely) in $\slt$ by taking 
\[t=\sma{rr} \sqrt{m}&0\\0&1/\sqrt{m}\fma\mbox{ and }
x=\sma{cc}1&1\\0&1\fma\]
unless $m=\pm 1$ in which case the matrix given for $t$ has finite order.
If $m=1$ then we have $\Z\times\Z$ which clearly embeds 
(both discretely and indiscretely) in $\slt$.
However $BS(1,-1)$ (the Klein bottle group) does not embed in $\slt$
because it is not CT (the group is torsion free, non abelian and with
$t^2$ in the centre so Lemma 2.1 applies). Thus we see that $BS(1,m)$
is CT if and only if $m\neq -1$. This is in disagreement with \cite{gkm}
where it is stated in Theorem 9 that $BS(1,m)$ is not CT for $m\neq 1$
(but no further comment is made as to why). However it is shown there that
$BS(1,m)$ is CSA only for $m=1$, as otherwise the group is non abelian but
has a non trivial normal abelian subgroup.        

There is a general case where the CSA and the CT condition are 
equivalent.
This occurs when $G$ is a group where the centraliser of each non trivial
element is infinite cyclic. Clearly $G$ is CT and it is not hard to show
that $G$ is also CSA, for instance see \cite{sha} Lemma 3.3.
We say that an element $g$ of infinite order in a group $G$
is {\bf primitive} if it generates its own centraliser.
We finish with some lemmas on the CSA property which will be needed later.
\begin{lem}
Let $G_1$ and $G_2$ be groups where the centraliser of every non trivial
element is infinite cyclic, and let $a$ be a primitive element of $G_1$ and
$b$ of $G_2$. 
Then any primitive element of $G_1$ or
$G_2$ is also primitive in the amalgamated free product
$G=G_1*_{a=b}G_2$.
\end{lem}
\begin{proof}
First suppose that 
$\gamma$ is in $G_1$ but cannot be conjugated into 
the amalgamated subgroup $H=\langle a\rangle$
by an element of $G$. 
We take $g\in G-G_1$, whereupon we can assume that
$g$ can be expressed as
$g_1g_2\ldots g_r$ for $r\geq 1$ 
and $g_1,g_2,\ldots ,g_r$ coming alternately from $G_1-H$ and $G_2-H$, which
we will refer to as a normal form for $g$. Conversely an element of this 
form is not in $G_1$ or $G_2$ if $r\geq 2$.
Now $g\gamma g^{-1}=g_1\ldots g_r\gamma g_r^{-1}\ldots g_1^{-1}$ 
is in normal form if
$g_r\in G_2-H$ because $\gamma\in G_1-H$. But if $g_r\in G_1-H$ then
$h_r=g_r\gamma g_r^{-1}$ is also in $G_1-H$ as $\gamma$ is not conjugate 
into $H$.
Moreover $r\geq 2$ in this case
so $g\gamma g^{-1}=g_1\ldots g_{r-1}h_rg_{r-1}^{-1}\ldots g_1^{-1}$ 
is in normal form and so is not in $G_1$.

We now need to consider the centraliser of elements in $H$.
If $h\in H-\{e\}$ and
$g$ is as above then 
$g_1\ldots g_{r-1}(g_rhg_r^{-1})g_{r-1}^{-1}\ldots g_1^{-1}$ is again in
normal form, because we have noted that
the condition on the centralisers being cyclic implies that
$H$ is malnormal in both $G_1$ and $G_2$, so
$ghg^{-1}\notin H$.
\end{proof}   

\begin{lem} If $G_1$ and $G_2$ are torsion free CSA groups
and $G=G_1*_HG_2$ where $H$ is a
centraliser in both $G_1$ and $G_2$ then $G$ is CSA.
\end{lem}
\begin{proof}
This can be proved using the concept of a 0-step malnormal 
amalgamated free product as defined in \cite{ks}, which means that the
amalgamated subgroup is malnormal in both factors. Alternatively we can
apply \cite{chmp} Appendix A which tells us when the fundamental group of
a graph of groups with abelian edge groups is CSA. In particular
Corollary A.8 states that the fundamental group $\pi_1(\Gamma)$ of a graph
of groups $\Gamma$ with CSA vertex groups and edge groups which are
maximal abelian in the neighbouring vertex groups is CSA if the
Bass-Serre tree of $\Gamma$ is acylindrical. This means that there is an
upper bound for the diameter of the fixed point set of any element in
$\pi_1(\Gamma)-\{e\}$. But for $\pi_1(\Gamma)=G_1*_HG_2$ we have that
a non identity element fixing two vertices which are a distance more than
one apart must fix two edges, thus is in $g_1Hg_1^{-1}\cap g_2Hg_2^{-1}$ for
$g_1^{-1}g_2\notin H$. However $H$ is malnormal in both $G_1$ and $G_2$,
thus is malnormal in $G$ (which follows in exactly the same way as Lemma 2.4
as we did not use the fact that centralisers were cyclic, just that the
amalgamated subgroup was malnormal in both factors).
\end{proof}
We give a quick example to show that Lemma 2.5 fails when we replace CSA
with CT. Let $G_1$ be the semidirect product $\Z^2\rtimes_\alpha\Z$ where
no power of $\alpha$ has fixed points. Then it is not hard to show that
$G_1$ is CT (indeed it follows from Theorem 5.1). Moreover the centralisers are
all infinite cyclic with the single exception of $\Z^2$. If we now let
$G_1=G_2$ with $H=\Z^2$ and form $G=G_1*_HG_2$ then $G$ has the
presentation
\[\langle s,t,a,b|sas^{-1}=\alpha(a),sbs^{-1}=\alpha(b),
tat^{-1}=\alpha(a),tbt^{-1}=\alpha(b)\rangle.\]
Now $\langle s,t\rangle$ generates a non abelian free group so $sts^{-1}t^{-1}$
and $s^2t^2s^{-2}t^{-2}$ do not commute, but they both commute with $a$.

Finally we need a version of Lemma 2.4 for HNN extensions.
\begin{lem}
Suppose that $G$ is a group where the centraliser of every non trivial
element is infinite cyclic, and that $a$ and $b$ are primitive elements
of $G$ where no conjugate of $A=\langle a\rangle$ intersects
$B=\langle b\rangle$ apart from in the identity.
Then any primitive element of $G$ 
is also primitive in the HNN extension
$\Gamma=G*_{tat^{-1}=b}$ where $t$ is the stable letter of the HNN extension.
Moreover if two primitive elements of $G$ are conjugate
in $\Gamma$ then either they are conjugate in $G$ or one is conjugate
in $G$ to $a^{\pm 1}$ and the other to $b^{\pm 1}$.
\end{lem}
\begin{proof}
We again have a normal form where we can write any 
element $\gamma\in 
\Gamma-G$ as $\gamma_0t^{k_1}\ldots t^{k_n}\gamma_n$ for $n\geq 1$,
$\gamma_i\in G$, $k_i\in\Z-\{0\}$ and no appearance of $ta^jt^{-1}$
or $t^{-1}b^jt$ occurs, and conversely an element in such form does
not lie in $G$.
First let us take a primitive element $g\in G$
which is not conjugate into $A$ or $B$. Thus for any $\gamma\in\Gamma-G$
expressed as above, we have that $\gamma_ng\gamma_n^{-1}$ is not in $A$ or
$B$. This implies that $\gamma g\gamma^{-1}$ is in normal form, thus
not in $G$. Now let us
consider a conjugate of $a$ by an element $\gamma$ outside $G$. Again 
on taking $\gamma$
in the form above, we see that $\gamma a\gamma^{-1}$ only fails to be
in normal form if $\gamma_n\in A$ and $k_n>0$, because $A$ is malnormal in
$G$ and we cannot have $\gamma_na\gamma_n^{-1}$ in $B$. Thus we can
replace $t\gamma_na\gamma_n^{-1}t^{-1}$ with $b$. But now this would only
fail to be in normal form if $k_n=1$, $\gamma_{n-1}\in B$ and $k_{n-1}<0$
which would contradict $\gamma$ being in normal form, unless $n=1$ when
$\gamma a\gamma^{-1}$ is conjugate to $b$.
\end{proof}

\section{Amalgamated free products}

Given an abstract group $G$ (always assumed to be countable and usually
finitely generated), we can ask whether $G$ is a linear group. Here we
will be concerned with a special version of this question:
does $G$ embed in $\slt$? We can rule out some
groups by noting that the finite subgroups of $\slt$ are very
restricted (for instance the only element of order 2 is the matrix $-I$)
so any group containing a finite subgroup not in this class
will fail to embed. However our focus in this paper is on torsion free
groups.

Another severe restriction comes because, as mentioned in the last section,
any subgroup of $\slt$ not containing $-I$ is CT. 
If a finitely generated torsion free group $G$ embeds as a discrete
subgroup of $\slt$, and hence a discrete subgroup of $PSL(2,\C)$,
then we know from hyperbolic geometry that  
$G=\pi_1(M)$ for $M$ an orientable 3-manifold with a complete
hyperbolic structure (or hyperbolic 3-manifold for short). Indeed by
topological tameness we now know that $M$ is homeomorphic to the
interior of a compact orientable 3-manifold. 
Conversely if $G=\pi_1(M)$ where $M$ is a compact orientable 3-manifold
with interior having a complete hyperbolic structure then $G$ is
finitely generated, indeed finitely presented, torsion free, and embeds
as a discrete group in $\slt$. Here we wish
to find examples of subgroups of $\slt$ with no discrete embedding.
There certainly are examples even amongst abelian groups: a discrete
torsion free abelian group can only be $\Z$ or $\Z\times\Z$ so that
if we take matrices $M_i= \sma{rr} 1&x_i\\0&1\fma$ where $x_1,\ldots ,x_m$ are
complex numbers which are linearly independent over $\Q$, the group
generated is $\Z^m$ for any $m$. A more interesting example gives
the wreath product of $\Z$ with itself, obtained by taking a diagonal
matrix with transcendental entries and $\sma{rr} 1&1\\0&1\fma$. This group
contains the free abelian group of countably infinite rank, and is
finitely generated but not finitely presented. These groups 
are certainly not word hyperbolic (because they contain $\Z^2$)
and (if they contain $\Z^4$)
cannot be the fundamental group of any (compact or non compact)
3-manifold, hyperbolic or not.
However they all sit in the subgroup of $\slt$ with bottom left hand
entry 0 and so must be metabelian. 

To obtain further examples which do not sit in some restricted subgroup,
we can try combining subgroups of $\slt$ in various ways. Using the
direct product is doomed to failure because $G_1\times G_2$ is not CT
unless both $G_1$ and $G_2$ are abelian (or one is trivial). Therefore
we should consider the free product, where there are known results on
linearity. Here we take a linear group to mean 
a subgroup of $GL(n,k)$ for $k$ a field of characteristic zero. 
Some similar results to those below
also exist for fields of positive characteristic but we will not discuss
them here. 

The free product $G_1*G_2$ of any two linear groups $G_1,G_2$
is linear: indeed it was proved in \cite{nis} back in 1940 that
if $G_1,G_2\leq GL(n,k)$ then $G_1*G_2\leq GL(n+1,k')$ for some other
field $k'$ of characteristic 0 and we can
replace $n+1$ with $n$ if there are no (non trivial) scalars in
$G_1$ or in $G_2$. This result also follows from Theorem 3 in \cite{weh}.
As $\C$ has uncountable transcendence degree over its prime field $\Q$ and 
all groups in this paper are countable, we can take
$k=k'=\C$ here without loss of generality.
Also Theorem 1 in \cite{sha} rediscovered a version of this result,
stating that if $G_1,G_2\leq SL(n,\C)$ with no scalar matrices in
either factor then $G_1*G_2$ embeds into $SL(n,\C)$. Therefore as an
immediate consequence of these results, we have
\begin{lem} 
If we have two countable subgroups $G_1$ and $G_2$ of $\slt$, neither
of which contain $-I$ then the free product $G_1*G_2$ embeds in $\slt$
as well. 
\end{lem}

Consequently an easy way of coming up with word hyperbolic groups
which embed in $\slt$ but with no discrete embedding
is to take $G_1=\pi_1(M_1)$ and $G_2=
\pi_1(M_2)$ where $M_1$ and $M_2$ are closed hyperbolic 3-manifolds. We
then have that $G_1*G_2$ embeds in $\slt$, is word hyperbolic (as the
free product of two word hyperbolic factors) and is the fundamental group
of a closed 3-manifold, the connected sum of $M_1$ and $M_2$. However it
cannot be the fundamental group of a closed hyperbolic 3-manifold or
even embed discretely in $\slt$: 
if $G=G_1*G_2$ is discrete 
then the resulting
quotient 3-manifold $M=\mathbb H^3/G$ must be irreducible, so 
it is determined by its
fundamental group up to homotopy. But the Mayer-Vietoris sequence for the
homology of a free product would give 
$H_3(G_1*G_2;\Z)=H_3(G_1;\Z)\oplus H_3(G_2;\Z)=\Z\oplus \Z$ which cannot
be equal to $H_3(\pi_1(M);\Z)$.

Given that linearity behaves well under free products, an immediate
question is whether this extends to free products with amalgamation
and/or HNN extensions. However there would need to be restrictions on the
factors/base and 
amalgamated/associated subgroups because we might not even have
residual finiteness in general. As we are concentrating here on torsion
free groups, the first candidates for amalgamated subgroups ought
to be those that are infinite cyclic. This is a fruitful pursuit for
$\slt$, as Proposition 1.3 in \cite{sha} shows:
\begin{prop}
Let $G_1*_HG_2$ be a free product amalgamating the subgroup $H_1\leq G_1$
with $H_2\leq G_2$ via the isomorphism $\phi:H_1\rightarrow H_2$.
Suppose that $G_1$ and $G_2$ are both subgroups of $\slt$ such that\\
(1) the matrices $\phi(h)$ and $h$ are the same for all $h\in H_1$,\\
(2) every $h\in H_1$ is a diagonal matrix, and\\
(3) for every $g_1\in G_1-H_1$ the bottom left hand entry is non zero, as
is the top right hand entry for all $g_2\in G_2-H_2$. 
Then $G_1*_HG_2$ can also be embedded in $\slt$. Further, let us say
that a subgroup of $\slt$ has transcendental traces if every non identity
element has a trace which is transcendental over $\Q$. Then if $G_1$ and
$G_2$ have transcendental traces, so does this embedding of $G_1*_HG_2$.  
\end{prop}

Earlier Theorem 5 in \cite{weh} proved that if $A$ and $B$ are free groups
with $a$ a primitive element of $A$
and $b$ a primitive element of $B$ then the amalgamated free
product $A*_{a=b}B$ embeds in $\slt$. In fact the proof there is very similar
to that of Proposition 1.3 as given in \cite{sha} by Shalen. However the
Shalen paper recognises that the condition of the factor groups being free
can be weakened considerably. Also in that paper Theorem 2 builds on this 
proposition to create a result on embedding amalgamated free products
in $\slt$ which can be applied recursively:
\begin{thm}
Let $A$ and $B$ be subgroups 
of $\slt$ with transcendental traces such that both groups satisfy the
following property: the centraliser of every
non identity element is infinite cyclic. Then the free product $A*B$,
as well as the amalgamated free product
$A*_{a=b}B$ for any primitive $a\in A$ and $b\in B$
has an embedding into $\slt$ with transcendental traces
and every non identity element has centraliser
which is infinite cyclic.      
\end{thm}

Here we would like to adapt these results somewhat to create a version
which does not require
the infinite cyclic restriction to hold for all centralisers as in 
Theorem 3.3.    

\begin{lem}
Suppose that $A$ is a group which embeds in $\slt$ such
that no element has trace $\pm 2$ except the identity. Suppose that 
$a\neq I$
is an element in $A$ of the form 
\[a=\sma{cl}\lambda&0\\0&\lambda^{-1}\fma.\] If any
element $x$ of $A$ has a zero entry then $x$ is also diagonal and so
commutes with $a$. 
\end{lem}
\begin{proof}
First suppose that 
\[x=\sma{rc}\alpha&0\\\gamma&1/\alpha\fma\qquad\mbox{ then }\qquad
axa^{-1}x^{-1}=\sma{cr}1&0\\(\lambda^{-2}-1)\gamma/\alpha&1\fma\]
which has trace 2 so $\gamma$ is zero. The same applies for a zero
in the bottom left hand entry. Now suppose that
\[x=\sma{cr}0&\beta\\-1/\beta&\delta\fma\qquad\mbox{ then }\qquad
xax^{-1}=\sma{cr}\lambda^{-1}&0\\(\lambda^{-1}-\lambda)\delta/\beta &\lambda
\fma\]
so we can now apply the above with $x$ replaced by $xax^{-1}$, telling us
that $\delta=0$ and forcing on us an element of trace 0, with square $-I$
in $A$.
\end{proof}

\begin{co}
Suppose that $A$ is a group which embeds in $\slt$ such
that no element has trace $\pm 2$ except the identity. Let $a\in A$
be a primitive
element in $A$ which has a transcendental trace under this embedding.
Let the same conditions hold for the group $B$ and the
element $b\in B$. Then the amalgamated free product $A*_{a=b}B$ can be
embedded in $\slt$ as well.
\end{co}
\begin{proof}
We first take conjugates in $\slt$ of $A$ and $B$ such that
\[a=\sma{rc}\lambda&0\\0&\lambda^{-1}\fma\qquad\mbox{ and }\qquad
 b=\sma{rc}\mu&0\\0&\mu^{-1}\fma\]
where $\lambda$ and $\mu$ will both be transcendental numbers. It is well
known (for instance see Lemma 3.2 of \cite{sha}) that
there exists a field automorphism $\psi$ of $\C$ such that
$\psi(\mu)=\lambda$. 
We apply $\psi$ to the entries of all elements of $B$
which will result in a subgroup of $\slt$ abstractly isomorphic to
(and still called) $B$ but such that $a=b$. Moreover this will not affect
which entries or traces of elements in $B$ are transcendental or equal to
$\pm 2$. If the bottom left hand entry of any element $x$ in $A$ is zero
then Lemma 3.4 says that $x$ is in the centraliser of $A$ and so in the
subgroup being amalgamated. The same will hold for all of the top right
hand entries of $B$, so we can now apply Proposition 3.2.
\end{proof}

We have already seen examples of subgroups of $\slt$ which are word hyperbolic
but have no discrete embedding. However they were fundamental groups of 
compact 3-manifolds, so we would now like examples which do not 
have this property. By \cite{bf} due to Bestvina and Feighn
we have that if $A$
and $B$ are torsion free word hyperbolic groups with $a\in A$ and $b\in B$
both non trivial elements then the amalgamated free product
$G=A*_{a=b}B$ is word hyperbolic if and only if $G$ contains no
$\Z\times\Z$ subgroup, which is shown to occur if and only if either 
$\langle a\rangle$ is malnormal in $A$ or $\langle b\rangle$ is malnormal
in $B$. This is also the same as saying that either $a$ or $b$ is
a primitive element in $A$ or $B$ respectively, because $A$ and $B$ are
both word hyperbolic, hence CSA.
(The result in \cite{bf} allows amalgamation over a virtually
cyclic subgroup but this is the statement when restricted to torsion free
groups.) 

Thus a particular example where Theorem 3.3 gives us word hyperbolic
groups embedding in $\slt$ is when $A$ and $B$ are
non abelian free groups.
Moreover we can take free groups $F_{r_1},\ldots ,
F_{r_n}$ for any ranks $r_i$ at least 2 and form repeated amalgamated free
products over arbitrary cyclic subgroups generated by primitive elements.
A well studied construction which is similar to this is that of a graph of
groups with non abelian free vertex groups and infinite cyclic edge groups.
The idea is that we obtain the fundamental group of a graph of groups by 
forming the amalgamated free product where an edge joins two distinct vertices,
then contract this edge and continue until we are left with self loops
which then give us HNN extensions. Consequently if the graph is a tree then
only free product amalgamation occurs. The Bestvina -Feighn result mentioned
above showing word hyperbolicity of such an amalgamated free product where
just one element is primitive does not extend to graphs of groups with this
property on the edge groups, for instance one can take three groups which
are free on $a,b$, on $c,d$ and on $x,y$ respectively, then amalgamating
$a^3$ to $c$ and $c$ to $x^3$ does not result in a word hyperbolic group.
However if we amalgamate elements which are primitive in each factor then
the resulting graph of groups is word hyperbolic by repeated use of
\cite{bf}, and Lemma 2.4 or a similar argument. Therefore by these two results
and Theorem 3.3 we have:
\begin{co}
If $\Gamma$ is a graph of non abelian free groups with infinite cyclic
edge groups which are generated by a primitive element on each side and
$\Gamma$ is a tree then the fundamental group of $\Gamma$ is word hyperbolic
and embeds in $\slt$.
\end{co}

We note that there exist examples where $\Gamma$ is not the fundamental group
of any 3-manifold and therefore cannot embed discretely in $\slt$. For
instance, as mentioned in the last section of \cite{gdwl}, we can take the
double $\Gamma$ over the primitive word $yx^2y^{-1}x^{-3}\in F_2$. If
$\Gamma=\pi_1(M)$ for $M$ a 3-manifold (assumed compact without loss of
generality) then we can take $M$ to be prime as $\Gamma$ is one ended.
This means here that $M$ is irreducible so the splitting over $\Z$ of
$\pi_1(M)$ can be induced geometrically, giving 3-manifolds $M_1, M_2$,
which are irreducible because $M$ is, 
with fundamental group free of rank 2. Now results in \cite{hem} imply
that $M_1$ and $M_2$ must each be a (possibly non orientable)
cube with one handle, joined to form
$M$ by attaching along a neighbourhood of
the curve $yx^2y^{-1}x^{-3}$ embedded in each
boundary. But we can attach a thickened disc along a
neighbourhood of one of these curves if it is an annulus
as in \cite{gdwl}, or attach a $P^2\times I$ if the neighbourhood is a
M\"obius strip. Thus we obtain a 3-manifold
with fundamental group $\langle x,y| yx^2y^{-1}x^{-3}\rangle$ 
or $\langle x,y| (yx^2y^{-1}x^{-3})^2\rangle$
which is well known to be a contradiction (for instance neither
Alexander polynomial is symmetric).

We consider the case of HNN extensions, where more care is needed
because these can introduce Baumslag-Solitar
subgroups, in Section 6, where we also mention other results on linearity
of these groups.

\section{Non Limit groups}
One definition of a limit group is that it is a finitely generated group
which is fully residually free. It is straightforward to show that 
such a group is torsion free and CT. It is also true that limit groups
embed in $\slt$, for instance see Window 8 of \cite{lubs} and we will provide
a quick and explicit proof in Section 6. Indeed 
\cite{ba67} by B.\,Baumslag shows that a finitely generated
residually free group is either a limit group or contains $F_2\times\Z$.
This last group is clearly not CT so limit groups are exactly the finitely
generated residually free groups which embed in $\slt$. Therefore we
must consider whether the groups given in this paper are limit groups, in
which case existence of an embedding into $\slt$ was already known.
First of all, as the only 
metabelian limit groups are the free abelian groups $\Z^m$, 
examples such as the wreath product given in Section 3 are not limit groups.
Moreover a
finitely generated subgroup of a limit group is also a limit group, but 
it was shown in \cite{wlt} using work of Sela that the fundamental group of
a closed orientable hyperbolic 3-manifold cannot be a limit group. 
Therefore the examples in the last section
which were formed using free products 
contain a finitely generated
subgroup which is not a limit group, meaning that they are not limit groups
either. The word hyperbolic but non 3-manifold group given at the end of the
last section is a limit group as it is a double of a free group,
but we can avoid this and still have a word
hyperbolic non 3-manifold group which embeds in $\slt$ by taking its free
product with a closed hyperbolic 3-manifold group.

Limit groups can contain $\Z^k$ for $k\ge 2$
and so need not be word hyperbolic but a limit group which does not
contain $\Z\times\Z$ is word hyperbolic. A special case of the
graph of free groups with $\Z$ edge groups mentioned in the last section
is the amalgamated free product
$G=F_{r_1}*_{w_1=w_2}F_{r_2}$ for $F_{r_1},F_{r_2}$ non abelian free groups
and $w_1,w_2$ any two non trivial elements. This is called a
{\bf cyclically pinched} group if neither $w_i$ is part of a free basis for
$F_{r_i}$ (otherwise we obtain the free group $F_{r_1+r_2-1}$). If neither
element is primitive, say $w_i=u_i^{n_i}$
for $n_i\ge 2$, then it is well known that $G$ cannot be a limit group.
A quick way to see this is to note that $G$ is not CT because
$u_1u_2u_1^{-1}u_2^{-1}$ is non trivial but $w_i$ commutes with $u_1$ and
$u_2$. However if we let one or both of the $w_i$ be primitive
elements, it is still not known exactly when $G$ is a limit group. 
Some cases are known, for instance doubles (where
$r_1=r_2$ and $w_1=w_2$) have been shown to be limit groups but if only
one element is primitive then examples of non limit groups go back to
Lyndon in 1959. He showed that any solution to the equation $x^2=y^2z^2$
in a free group has the property that $x,y$ and $z$ all commute. This
means that if $G=F_{r_1}*_{u^2=v^2w^2}F_{r_2}$ for $u,v,w$ any non trivial
words then any homomorphism from $G$ onto a free group $F$ sends $u,v,w$
to commuting elements in $F$. Thus if $v$ and $w$ do not commute in $F_{r_2}$
then all homomorphisms from $G$ to a free group send the commutator
$[v,w]$ to the identity. Moreover the same property was shown in \cite{lyscz}
to
hold for the equation $x^l=y^mz^n$ for $l,m,n\ge 2$ so that an element
of the form $y^mz^n$ is not a proper power in a free group if $y$ and $z$
do not commute, and thus $G=F_{r_1}*_{u^l=v^mw^n}F_{r_2}$ is a word
hyperbolic group which is not residually free.

Let us now assume that both $w_1$ and $w_2$ are primitive elements.
If the process of forming repeated
amalgams of free groups over maximal cyclic subgroups always resulted in a   
limit group then Corollary 3.8 would follow immediately. In fact this is 
not true although we are unaware of specific examples in the
literature so will give the first ones here. Our reference on this question is
\cite{sel} where it is mentioned in Part I (3)
that ``if the element $w_1$ is a commutator
in the first free group and $w_2$ is a product of two ``high'' powers in the
other free group then $G$ is not a limit group'' but no further details
are given. However this is not true in full generality, even for any
definition of ``high'': if $G_1$ is free on $x,y$; $G_2$ free on $a,b$ and
$G=G_1*_{w_1=w_2}G_2$ is formed by setting $[x,y]=a^mb^n$ for $m,n\neq 0$ 
with highest common factor $d$ then
consider the homomorphism from $G$ onto the free group $F(t,u)$ on $t,u$
given by
$x\mapsto t,y\mapsto u^{mn/d}$ and $a\mapsto tu^{n/d}t^{-1}, 
b\mapsto u^{-m/d}$. This sends both rank 2 free groups $G_1$ and $G_2$ to
rank 2 free subgroups of $F(t,u)$, thus the restriction to each $G_i$ is
injective. Consequently $G$ is an example of a generalised double over 
$F(t,u)$ as in Definition 4.4 of \cite{chmp}, with Proposition 4.7 of
\cite{chmp} showing that a generalised double over a limit group is also
a limit group.

In fact the quote above becomes true if we change ``powers'' to
``powers of commutators'' with a suitably weak definition of ``high''.
To prove this we will use stable
commutator length, as described in \cite{cal}. This can be explained
briefly as follows: a {\bf length} on a group $G$ is a function
$l:G\rightarrow\R$ such that for $g,\gamma\in G$ we have
\[l(g\gamma)\le l(g)+l(\gamma)\mbox{ and }l(g)=l(g^{-1}).\]
(Sometimes $l(e)=0$ is required but this will not affect any of our
results here.) From any length function $l$ we obtain {\bf stable length}
$\sigma:G\rightarrow\R$ defined by $\sigma(g)=\lim_{n\rightarrow\infty}
l(g^n)/n$. As for any $g\in G$ the sequence $a_n=l(g^n)$ is subadditive
(meaning that $a_{n+m}\le a_m+a_m$ for $n,m\in\N$), this limit exists
provided only that there is $K\le 0$ with $a_n\ge Kn$ for all $n$. It is
straightforward to show using only the above properties that $\sigma$
is constant on conjugacy classes, that $\sigma(g^k)=|k|\sigma(g)$ for all
$k\in\Z$ and $g\in G$, and that $\sigma(gh)\le\sigma(g)+\sigma(h)$ if
$g$ and $h$ commute (although not in general). Given a group $G$, 
{\bf commutator length} cl is a length on the commutator subgroup 
$G'$ of $G$ with
cl$(g)$ defined to be the minimum number of
commutators needed to form a product equal to $g$ and
{\bf stable commutator length} is defined to be the stable length
that results, denoted by scl$(g)$. A non trivial fact about scl
in free groups $F$ which we will need here is that
every non identity element $w\in F'$ has scl$(w)\ge 1/2$
by \cite{cal} Theorem 4.111. Combining this with the point that in
general a
commutator $[w_1,w_2]$ has scl$([w_1,w_2])\leq 1/2$, we see that non
trivial commutators in free groups have stable commutator length
exactly $1/2$.
\begin{thm} If $F_1$ and $F_2$ are free non abelian groups then the
cyclically pinched group $G=F_1*_{\gamma=\delta^m\eta^n}F_2$, where
$\gamma$ is a non trivial commutator in $F_1$ and $\delta,\eta$ are both
non trivial commutators in $F_2$ with $\delta\neq\eta^{\pm 1}$,
is not residually free whenever
$|m|-|n|\ge 3$ and $|m|,|n|\neq$ 0 or 1, even though $G$ embeds in $\slt$.
\end{thm}
\begin{proof}
We can assume by changing elements to inverses that $m,n>0$.
We know that $\gamma\neq e$ in $G$ so suppose we have a homomorphism
$\theta$ from $G$ to a free group $F$ where $\theta(\gamma)=a\neq e$. This
gives us a non trivial commutator in $F$ which is equal to $b^mc^n$ for
two (possibly trivial) commutators $b=\theta(\delta)$ and
$c=\theta(\eta)$ in $F$. First suppose that $b$
is trivial in $F$ then $a=c^n$ but in a free group a non trivial
commutator cannot be a proper power. (This is due to Sch\"{u}tzenberger
but can also be seen here because in a free group 
scl$(a)=1/2=|n|\mbox{scl}(c)\geq |n|/2\geq 1$, using the fact that $c$
must be in the commutator subgroup $F'$ too.) The same holds if $c$ is trivial.
As $\gamma$ is primitive and we mentioned that $\delta^m\eta^n$ is also
primitive, Corollary 3.8 tells us that $G$ embeds in $\slt$.
Therefore we are done by the next Proposition.
\end{proof}
\begin{prop}
If the equation $a=b^mc^n$ holds in a free group $F$ where $a,b,c$ are all 
non trivial commutators in $F$ then we cannot have
$|m|-|n|\ge 3$.
\end{prop}
\begin{proof}
By Theorem 2.70 in \cite{cal} using Barvard
duality, we have that for any $g$ in $G'$, 
\[\mbox{scl}(g)=\frac{1}{2}
\mbox{sup}_{\phi\in Q(G)}\frac{|\phi(g)|}{|D(\phi)|}\] 
where $Q(G)$ is the
space of homogeneous quasimorphisms on $G$ and $D(\phi)$ is the defect of 
$\phi$. This means that 
for all $a,b\in G$ and $k\in\Z$ we have
$|\phi(ab)-\phi(a)-\phi(b)|\le D(\phi)$ (a quasimorphism) and 
(homogeneity) $\phi(a^k)=k\phi(a)$.

Thus we can take a homogeneous quasimorphism $\phi_b\in Q$ with
$|\phi_b(b)|/D(\phi_b)$ arbitrarily close to 1 as $\mbox{scl}(b)=1/2$.
By rescaling, let us say that
$D(\phi_b)=1$ and we can assume $m\phi_b(b)>n+2$ because $m-n\ge 3$.
(In fact in a free group this supremum is obtained but we will not
need to use that
here.) Now for any homogeneous
quasimorphism $q$ on any group $G$ we have that
$|q(\gamma)|\le D(q)$ if $\gamma$ is a commutator in $G$. But as $a=b^mc^n$
holds in $F$ we see that
\[m\phi_b(b)-n-1\le
m\phi_b(b)-n|\phi_b(c)|-|\phi_b(a)|\leq|m\phi_b(b)+n\phi_b(c)
-\phi_b(a)|\le 1\]
so $m\phi_b(b)\le n+2$, giving us a contradiction.
\end{proof}    
\section{2 Dimensional linearity of 3-manifold groups}
Despite Perelman's positive solution to Thurston's Geometrisation
Conjecture, which as a consequence establishes the fact that all
finitely generated 3-manifold groups are residually finite, it is still
an open question as to whether all such groups are linear. Indeed in
\cite{kir} Problem 3.33 (A) Thurston asks whether all such groups can be
embedded in $GL(4,\R)$, which was unknown until just now \cite{btnx}.
There has been very recent progress, using results 
of Wise on virtually special groups, which shows that the fundamental group
$\pi_1(M^3)$ of most compact 3-manifolds $M^3$ has a finite index subgroup 
which embeds in a right angled Artin group, thus this subgroup and hence
also $\pi_1(M^3)$ will be linear over $\Z$, although the dimension might be 
very big. Indeed following the release of \cite{liu} and \cite{ppw}, the 
only compact orientable irreducible 3-manifolds not having virtually special  
fundamental group are those closed Siefert fibre spaces and those closed
graph 3-manifolds which do not admit a Riemannian metric of non positive
curvature, but the former are known to be linear over $\Z$.
 
Here we will restrict ourselves to examining which closed 3-manifolds have
fundamental groups that embed in $SL(2,\mathbb C)$, in addition to those
that embed discretely. Our first result is for torus bundles, none of
which have fundamental groups that embed discretely.
\begin{thm} If $M^3$ is a closed 3-manifold which is a 2 dimensional
torus bundle over the circle, so that $\pi_1(M^3)=\Z^2\rtimes_{\alpha}\Z$
where $\alpha$ is the automorphism of $\Z^2$ induced by the gluing map, then
$\pi_1(M^3)$ is a subgroup of $\slt$ if and only if $\alpha$ is the
identity or is a hyperbolic map (that is all powers of $\alpha$ fix only 0).
\end{thm}
\begin{proof} Let $G=\pi_1(M^3)$ and $\Z^2=\langle a,b\rangle$, with
conjugation by the stable letter $t$ giving our automorphism $\alpha$.
If this is the identity then $G=\Z^3$ embeds in $\slt$. If
some positive power $\alpha^n$ fixes $x\in\Z^2-\{0\}$ and $G$ embeds in
$\slt$ then $t$ commutes with $x$ because $G$ is CT, but $x$ commutes
with all of the fibre subgroup $\Z^2$ so $t$ does too and so we have the
identity automorphism.

Otherwise the eigenvalues of $\alpha$ are not roots of unity and so
not of modulus 1, because they satisfy a monic integer quadratic with
constant term $\pm 1$.
Say $\alpha(a)=tat^{-1}=a^ib^j$ and $\alpha(b)=tbt^{-1}=a^kb^l$. We set
\[a=\sma{rr}1&1\\0&1\fma,\qquad b=\sma{rr}1&x\\0&1\fma\qquad\mbox{and }
t=\sma{cc}\mu&0\\0&\mu^{-1}\fma\]
where $x$ and $\mu$ are complex numbers to be determined. For the above
two relations to be satisfied, we require $\mu^2=i+xj$ and $\mu^2 x=
k+xl$. This just corresponds to the matrix
$\sma{rr}i&j\\k&l\fma\in GL(2,\Z)$ having an eigenvalue $\mu^2$ with the
eigenvector $\sma{c}1\\x\fma$ which does occur, and $x$ is
not zero because $\sma{c}1\\0\fma$
or $\sma{c}0\\1\fma$ being an eigenvector implies that 
there was an eigenvalue of $\pm 1$. Moreover $x\in\Q$ implies that
$\mu^2\in\Z$ but the determinant being $\pm 1$ would give
$\mu^2=\pm 1$ too, which has been eliminated. Therefore $\mu^2$
not being a
root of unity implies that the matrix $t$ has infinite order, and the
matrix $a^pb^qt^r=\sma{cc}\mu^r&\mu^{-r}(p+xq)\\0&\mu^{-r}\fma$ can
only be the identity if $p=q=r=0$, so this
representation of $G$ in $\slt$ is faithful.\end{proof}
Remarks: (1) It did not matter here if $\alpha$ was orientation preserving
or reversing and hence if $M^3$ were orientable or not.\\
(2) This shows just how much variation there can be in the geometry of
a 3-manifold and still the fundamental group embeds in $\slt$. For
instance we can take the connected sum of a Lens space (with finite
cyclic fundamental group), the 3-torus (with fundamental group
$\Z^3$), a closed orientable hyperbolic 3-manifold and a torus bundle
with a hyperbolic automorphism as in Theorem 5.1. The fundamental
group of the resulting 3-manifold embeds in $\slt$ by Lemma 3.1 but there
are pieces of this manifold which are positively curved, have zero
curvature, have negative curvature and which do not admit a metric of
nonpositive or of nonnegative curvature.

Based on these examples, one might make a naive conjecture as to when the
fundamental group $G=\pi_1(M^3)$ of a closed orientable irreducible
3-manifold $M^3$ embeds in $\slt$
as follows: first suppose that $M^3$ admits one of the
eight 3 dimensional geometries. We go work through
the groups obtained from manifolds having $S^3$ 
geometry and list those which embed in $\slt$ 
and we do the same with manifolds having  
the Sol (virtually soluble) geometry (where all groups will have a finite
cover of the form in Theorem 5.1 with hyperbolic gluing map).
Then of course all groups coming from manifolds with the $\HH^3$
geometry embed, but we obtain no further groups from the other geometries
apart from $\Z$ and $\Z^3$. 
This is because now $\pi_1(M^3)$ must be infinite and torsion free. Either
we have the $S^2\times\R$ geometry (giving virtually cyclic groups), 
the $\E^3$ geometry (virtually abelian), or
we have Nil (virtually nilpotent groups), $\HH^2\times\R$
or $\tilde{PSL}(2,\R)$ geometries. The last two cases give rise to
Siefert fibred spaces so here the fundamental group will have a finite
cover having a normal infinite
cyclic subgroup $\langle x\rangle$. This means that $gxg^{-1}=x^{\pm 1}$ 
for all group elements $g$ in the finite cover, so 
the centraliser of $x$ has finite index in the fundamental group.
Consequently Lemma 2.1 applies to all five of these geometries.

This leaves us with closed orientable irreducible 3-manifolds which do
not possess a geometry but which can be cut along embedded incompressible
tori to obtain hyperbolic 3-manifolds with torus boundary and Seifert
fibre spaces with torus boundary. Again we
have by Lemma 2.1 that the fundamental group of a Seifert fibre space
cannot embed in $\slt$ unless it is abelian. As we have non empty boundary
here, this leaves only $\Z$ and $\Z\times\Z$. Thus if there is a
Seifert fibre space in the torus decomposition of $M^3$ then its fundamental
group will be a subgroup of $\pi_1(M^3)$ which cannot therefore embed in
$\slt$. Therefore we might finish off our conjecture by hoping that if
all the pieces in the torus decomposition of $M^3$ are hyperbolic (and
hence embed in $\slt$) then
$\pi_1(M^3)$ is a subgroup of $\slt$. However this turns out not to be
the case, even if $\pi_1(M^3)$ is a CSA group. 
\begin{thm}
Let $M_1^3,M_2^3$ be two copies of the figure eight knot complement
with each boundary
$\partial M_1^3,\partial M_2^3$ a torus, and let $M^3$ be the closed
3-manifold formed by gluing the boundary tori together using any
orientation preserving homeomorphism which identifies the two meridians,
with the exception of the homeomorphism which also identifies the two
longitudes.
Then $\pi_1(M^3)$ does not embed in $\slt$ 
although it is CT and even a CSA group.
\end{thm}
\begin{proof}
Although $M_1^3$ has only one discrete faithful embedding in $\slt$
(because it has finite hyperbolic volume so Mostow rigidity 
applies), up to conjugation in $\slt$, replacing matrices with their
negative and taking complex conjugates,
there is a whole curve of representations. This curve was found in
\cite{whi} where it was shown that if $A$ and $B$ are two 
non commuting
elements of
$\slt$ satisfying
\[r(A,B)=B^{-1}A^{-1}BAB^{-1}ABA^{-1}B^{-1}A=I\]
then we can conjugate $A$ and $B$ so that either
\begin{equation}
A=\sma{cc}1&1\\0&1\fma\mbox{ and }B=\sma{rc}1&0\\-\omega&1\fma\mbox{ where }
\omega=e^{2\pi i/3}\mbox{ or }e^{-2\pi i/3}
\end{equation}
(or $A$ and $B$ are both minus the above)
or we can take
\begin{equation}
A=\sma{cl}\lambda&0\\0&\lambda^{-1}\fma\mbox{ and }B=\sma{cc}
\mu&1\\\mu(x-\mu)-1&x-\mu\fma
\end{equation}
(or both minus this) where 
$\lambda\in\C-\{0,\pm 1\}$, $x=\lambda+\lambda^{-1}$,
\[z=\frac{1}{2}\left(1+x^2\pm \sqrt{(x^2-1)(x^2-5)}\right)\mbox{ and } 
\mu=\frac{\lambda z-x}{\lambda^2-1}.\]

The figure eight knot complement
can also be thought of as the once punctured torus
bundle given by taking the free group of rank 2 on $a,b$ and forming the
HNN extension
\[\langle t,a,b|tat^{-1}=aba,tbt^{-1}=ba\rangle\]
with stable letter $t$. Then $\langle t,aba^{-1}b^{-1}\rangle=\Z\times\Z$
which forms the boundary torus with $t$ the meridian and $aba^{-1}b^{-1}$
the longitude. That this is isomorphic to the group 
$\langle A,B|r(A,B)\rangle$ given above can be seen by setting $a=BA^{-1}$,
$b=B^{-1}ABA^{-1}$ and $t=A$, so that $A=t$ and $B=at$ which is also equal
to $b^{-1}tb$, explaining why the matrices $A$ and $B$ given above are
conjugate in $\slt$. This means that the meridian $m$ is equal to
$A$ and the longitude is $l=BA^{-1}B^{-1}A^2B^{-1}A^{-1}B$.

We first make a couple of observations about the form of the matrices
in (2) in order to exploit the symmetries that are present. First
note that if we replace
$\lambda$ with $\lambda^{-1}$ in (2) then the top left and bottom right
hand entries of $A$ are swapped whereas the other entries stay the same, 
and also for $B$. However we can now conjugate the new $A$ and $B$ by
$\sma{cc}0&\nu\\-\nu^{-1}&0\fma$ for an appropriate choice of $\nu$
so that they become the original $A$ and $B$. Thus we can change $\lambda$
to $1/\lambda$ if desired because they both represent the same parametrisation
up to conjugacy. Now the longitude $l$ commutes with $A$ and 
so should be of the form $\sma{cc}f(\lambda)&0\\0&f(\lambda)^{-1}\fma$ for
some function $f$, because the entries of any element in this group depend only
$\lambda,z,x$ and $\mu$, with the last three variables all functions of 
$\lambda$. But there appears to be an ambiguity in that plus or minus
a square root is taken in the definition of $z$, which would result in two
possible functions $f_+(\lambda)$ or $f_-(\lambda)$ giving rise to the entries
on the diagonal of $l$. 

However we can use the standard identities to
find the trace $\tau$ of $l$ in terms of the traces of $A,B$ and $AB$, which 
here are equal to $x,x$ and $z$ respectively. The answer is
$\tau(x,x,z)=2+x^2(z-2)(x^4+(z+2)(z+2-2x^2))$ but we also have the equation
$z^2=(1+x^2)z-2x^2+1$ holding which is quadratic in $z$. If we now
substitute this expression for $z^2$ into $\tau$ we find that
$z$ vanishes, so we are left with the trace of $l$ being a function of 
$x$ only, hence of a function of $\lambda$ only too.

We now suppose that we have an embedding of $G=\pi_1(M^3)$ in $\slt$ where
$G=G_1*_{\Z\times\Z}G_2$ for $G_i=\pi_1(M_i^3)$ and $\Z\times\Z$ is the
fundamental group of the
boundary torus. We can conjugate $G$ in $\slt$ so that without loss of
generality we have $G_1=\langle A_1,B_1\rangle$
with $A_1,B_1$ matrices in the form above, giving rise to the meridian
$m_1=A_1$ and the longitude $l_1(A_1,B_1)$. 
First suppose that $A_1$ and
$B_1$ are of the form in (2) for some parameter $\lambda_1\neq 0,\pm 1$
(and other parameters $x_1,z_1,\mu_1$ depending on $\lambda_1$), so that
the longitude $l_1(A_1,B_1)$ will also be a diagonal
matrix $\sma{cl}d_1&0\\0&d_1^{-1}\fma$ say, with 
$\tau(\lambda_1)=d_1+d_1^{-1}$. 
We also have the meridian $m_2$ and
longitude $l_2$ of $G_2$ and we are forcing
$m_2$ to be equal to $m_1^{-1}=A_1^{-1}$, so that the two figure 8 knot
complements are joined on either side of the boundary torus.
Moreover the homeomorphism must identify the
longitude $l_2$ of $M_2^3$ with the curve $m_1^nl_1$ in $\partial M_1^3$
to obtain an orientation reversing homeomorphism between the two boundary 
tori so as to match the orientations of the 3-manifolds. 
In particular the longitude $l_2$ of $G_2$ must be a diagonal matrix
$\sma{cl}d_2&0\\0&d_2^{-1}\fma$ say, because it commutes with $m_2=m_1^{-1}$.

Now although the group $G_2$ has a fixed embedding in $\slt$ because $G$ has
been conjugated in order to put $G_1$ into a suitable form, we 
can also separately conjugate $G_2$, by $X\in\slt$ say, so that $XG_2X^{-1}
=\langle A_2,B_2\rangle$ where again $A_2,B_2$ are as
in (2) though this time with the parameter $\lambda_2$. But 
$X^{-1}A_2^{-1}X=m_2^{-1}=A_1$ is equal
to $\sma{cc}\lambda_1&0\\0&\lambda_1^{-1}\fma$ and the only way that
$A_2^{-1}=\sma{lc}\lambda_2^{-1}&0\\0&\lambda_2\fma$ can be conjugate 
to this is if $\lambda_2=\lambda_1$ or $\lambda_1^{-1}$. 
Therefore we can assume by the above comment that $\lambda_2$ is equal to
$\lambda_1$. But the longitude in $\langle A_2,B_2\rangle$ is the
element $Xl_2X^{-1}$ which will have trace $\tau(\lambda_2)=\tau(\lambda_1)$.
Thus the trace $d_2+d_2^{-1}$ of $l_2$ is also $\tau(\lambda_1)=d_1+d_1^{-1}$,
giving us that $d_2=d_1^{\pm 1}$. But as $l_2=m_1^nl_1$ is a product
of diagonal matrices, we obtain $d_2=\lambda_1^nd_1$. Thus we either
have $d_1^2\lambda_1^n=1$, which is a contradiction because 
$\langle l_1,m_1\rangle=\Z\times\Z$, or $\lambda_1^n=1$ which is also a
contradiction unless $n=0$. 

The case in (1) is similar but quicker. We can assume by conjugating
$G$ in $\slt$, as well as taking minus signs and complex conjugation
if necessary, that $A_1=A$ and $B_1=B$ in (1) for $\omega=e^{2\pi i/3}$
so a quick calculation tells us that $l_1=\sma{cc}-1&-2\sqrt{3}i\\
0&-1\fma$. 
Also we again have $X\in\slt$ so that $XG_2X^{-1}=\langle A_2,B_2\rangle$
for $A_2,B_2$ as in (1) but with $\omega=e^{\pm 2\pi i/3}$ (we can rule
out having to put minus signs in front of $A_2$ and $B_2$
because $A_2=XA_1X^{-1}$ so $A_2$ has trace 2).
Thus $Xl_2X^{-1}$ is equal to $l_1$ or its complex conjugate.
Now $l_2=m_1^nl_1=\sma{rc}-1&-2\sqrt{3}i-n\\0&-1\fma$ but we know 
$A_2=m_2=m_1^{-1}=A_1^{-1}$ so
$X$ conjugates $A_1$ into its inverse and therefore 
can only be $\pm\sma{cr}i&0\\0&-i\fma$. But on comparing $Xl_2X^{-1}$
and $l_1$, we again see that $n$ can only be zero.

As for the last part, it is well known that discrete torsion free subgroups
of $\slt$ are CSA groups.
For instance this can be seen quickly by using
Corollary 2.3 along with the straightforward fact that any group
containing $\pm\sma{cc}1&x\\0&1\fma$ for $x\neq 0$ and 
any infinite order element of the form
$\sma{cl}\lambda&y\\0&\lambda^{-1}\fma$ for $\lambda\neq 0,\pm 1$ and
$y\in\C$ is non discrete. Thus $G_1$ and $G_2$ are CSA, so we are done
by Lemma 2.5.
\end{proof}

In the case where the 
longitudes are glued to each other
we do in fact have an embedding 
of $G$ in $\slt$, giving an example of 
a closed orientable 3-manifold $M$ fibred over the circle with fibre
a genus 2 surface and with a non pseudo-Anosov gluing homeomorphism
but such that $\pi_1(M)$ still embeds in $\slt$. This
follows from Proposition 3.2, provided we can show that there are
faithful embeddings of the fundamental group of the figure 8 knot
complement of the form (2) in Theorem 5.2:
\begin{prop}
For any transcendental number $\lambda\in\C$, the matrices in (2)
provide a  
faithful embedding in $\slt$
of the fundamental group of the figure 8 knot
complement $G$.
\end{prop}
\begin{proof}
On taking $\lambda$ to be any transcendental number, we have
that $x$ and $z$ will be transcendental too (for
either choice of $z$). Now it is well known that
the trace of any element in $\langle A,B\rangle$ is given by
a triple variable polynomial in the trace of $A$, of $B$, and of $AB$, 
having coefficients in $\Z$. Therefore if there is an element $g$ in $G$
which has trace $\pm 2$ for this value of $\lambda$ then we know that 
some polynomial $f$ in $\Z[u,v]/\langle r(u,v)\rangle$ is zero at 
$(u,v)=(x(\lambda),z(\lambda))$ for one of the choices of $z$,
where $r(u,v)$ is the irreducible
polynomial $v^2-(1+u^2)v+2u^2-1$. Using this relation
we can assume that $f(u,v)$ is of the form
$vp(u)+q(u)=0$ for $p,q\in\Z[u]$ which implies that
$q^2-p^2+2u^2p^2+(1+u^2)pq$ is zero when evaluated at $x$.
As $x$ is transcendental, this polynomial
must be identically zero and so 
$(pz+q)(pz-p(1+x^2)-q)=0$ holds for all values of $x$ and
$z$ satisfying $f(x,z)=0$. Thus it must be a multiple of $r(x,z)$, which 
is an irreducible degree 2 polynomial in $z$, giving
a contradiction unless $p=q=0$. Hence
the trace of $g$ is constant in all representations.
But on setting $x=2$ so that $z=(5\pm\sqrt{-3})/2=
2-\omega$, we have the faithful discrete representation in (1).
In this case we know
that the elements with trace $\pm 2$ 
can all be conjugated to lie in the 
$\Z\times\Z$ subgroup $\langle m(A,B),l(A,B)\rangle$. Now $m(A,B)=A$ and
$l(A,B)$ are both diagonal matrices in all other representations we are
considering. Hence if $m^il^j$ has (without loss of generality)
trace equal to 2 when $\lambda$ is transcendental, this element must be
the identity and hence the identity for any $\lambda$. But on putting
say $\lambda=2$ we find that this cannot hold unless $i=j=0$.
\end{proof}

\section{Embeddings of HNN extensions in $\slt$}

Having seen cases where we can embed free products amalgamated over $\Z$
in $\slt$ as long as the factors have this property, we now see that
this is very different with HNN extensions. This is perhaps not surprising,
given that if we form an HNN extension of $G$ with stable letter $t$ and
isomorphic
associated subgroups $A$ and $B$ then we introduce relations of the form
$tat^{-1}=b$ for $a\in A$ and $b\in B$, forcing $a$ and $b$ to have the
same trace.

We start with a result giving conditions on when HNN extensions do embed
in $\slt$, in analogy with Corollary 3.5.
\begin{thm} Suppose that $G$ is a countable subgroup of $\slt$ such that
no element of $G$ has trace $\pm 2$ except the identity (so that $G$ is a
CSA group by Corollary 2.3). Let $A\neq G$ be any centraliser in $G$ and
$B=gAg^{-1}$ any centraliser that is conjugate to $A$ by $g\in G$. 
Then the HNN extension
$\Gamma=G*_{tAt^{-1}=B}$, formed using the isomorphism from $A$ to $B$ which is
conjugation by $g$, embeds in $\slt$ and no non-identity element of
$\Gamma$ has trace equal to $\pm 2$.
\end{thm}
\begin{proof}
We are forming the group
\[\Gamma=\langle G,t|tat^{-1}=gag^{-1}\,\,\forall a\in A\rangle\]
which is isomorphic to the HNN extension
\[\langle G,s|sas^{-1}=a\,\,\forall a\in A\rangle\]
by setting $s=g^{-1}t$, so we can assume our HNN extension is formed using
the identity map.

We now take $a\in A-\{e\}$ and conjugate $G$ in $\slt$ so that $a$ is a
diagonal matrix with distinct eigenvalues. This means that
$A$ is now precisely the subgroup $D$ of diagonal
elements in $G$. Thus by Lemma 3.4 no element of $G-D$ has a
zero entry.
We set $t=\sma{cl}x&0\\0&x^{-1}\fma$ where $x\in\C$ is transcendental over
the coefficient field $\F$ generated by the elements of $G$. Now $t$
commutes with all elements of $A$ so we need to show that 
$\langle t,G\rangle\le\slt$ is a faithful representation of the HNN
extension $\Gamma$. By using normal forms, any $\gamma\in\Gamma$ which is
not a conjugate of an element of $G$ (or a power of $t$) 
can be conjugated so it
is of the form
\[\gamma=t^{n_1}g_1t^{n_2}g_2\ldots t^{n_r}g_r\mbox{ for }r\ge 1,
n_i\in\Z-\{0\}\]
and where each $g_i$ is in $G-A$.

Now each entry of $\gamma$ is a Laurent polynomial in $x$, involving
positive and negative powers, with coefficients in $\F$. We set
$n$ to be $|n_2|+\ldots+|n_r|$ which is greater than 0 for $r\geq 2$
as well as $S=n_1+n,D=n_1-n$
and we claim by induction on $r$ that
$\gamma$ must be of the form
\[\sma{ll}
a_+x^S+\ldots +u_+x^D&a_-x^S+\ldots +u_-x^D\\
c_+x^{-D}+\ldots +v_+x^{-S}&c_-x^{-D}+\ldots +v_-x^{-S}\fma\]
where $a_{\pm},c_{\pm},u_{\pm},v_{\pm}$ are all non zero elements of $\F$
and we have only written down the end terms (the highest and lowest powers
of $x$) of each Laurent polynomial. This is because on increasing $r$ by 1,
we obtain a new element by multiplying $\gamma$ above by
\[t^{n_{r+1}}g_{r+1}=\sma{ll}\alpha_+x^m&\beta_+x^m\\
\alpha_-x^{-m}&\beta_-x^{-m}\fma\] 
where we have set $m$ equal to $n_{r+1}$ and
$\alpha_\pm,\beta_\pm$ are the entries of $g_{r+1}$ so are again
non zero elements of $\F$.

Thus 
our new element $t^{n_1}g_1t^{n_2}g_2\ldots t^{n_r}g_rt^{n_{r+1}}g_{r+1}$
is equal to
\[\sma{ll}a_\pm\alpha_\pm x^{S+|m|}+\ldots+u_\mp\alpha_\mp x^{D-|m|}&
a_\pm\beta_\pm x^{S+|m|}+\ldots+u_\mp\beta_\mp x^{D-|m|}\\
c_\pm\alpha_\pm x^{-D+|m|}+\ldots+v_\mp\alpha_\mp x^{-S-|m|}&
c_\pm\beta_\pm x^{-D+|m|}+\ldots+v_\mp\beta_\mp x^{-S-|m|}\fma\]
where we take the upper signs throughout if $m>0$ and the lower for $m<0$.
Hence the induction is
established and we see that each entry is a Laurent polynomial in $x$ with
more than one non zero term, so each entry is transcendental. Moreover the
trace of $\gamma$ is a Laurent polynomial of the form
$a_+x^{n_1+n}+\ldots+v_-x^{-n_1-n}$ if $n_1>0$ and
$c_-x^{-n_1+n}+\ldots+u_+x^{n_1-n}$ otherwise, thus all traces of elements
outside $G$ are transcendental.
\end{proof}

Consequently Theorem 6.1 can be applied iteratively, thus allowing repeated 
HNN extensions which will always result in a CSA group that embeds in
$\slt$. As pointed out by H.\,Wilton, this gives a quick and explicit
proof that all limit groups embed in $\slt$: an free extension of centralisers
of a group $G$ is the group $G*_C(C\times A)$ where $C$ is the centraliser
of an element in $G$ and $A=\Z^k$ is a finitely generated free abelian
group. Theorem 4 of \cite{km98b}, with an alternative proof given in 
\cite{chmp} as Theorem 4.6, states that a finitely generated group is a
limit group if and only if it is a subgroup of a group
formed by a finite number of
repeated free extensions of centralisers, starting from a 
finitely generated free group.
However we can write
\[G*_C(C\times\Z^2)=(G*_C(C\times\Z))*_{C\times\Z}
((C\times\Z)\times\Z)\]
and so on. Although $C\times\Z$ need not be a centraliser in the new group
$\Gamma=G*_C(C\times\Z)$, this is true if $C$ is malnormal in $G$ and then 
$C\times\Z$ is
malnormal in $\Gamma$. Therefore if $G$ is a CSA group
we need only apply a rank 1 
free extension of centralisers $k$ times to go from $G$ to 
$G*_C(C\times \Z^k)$. But a rank 1 extension is simply an HNN extension
$G*_{tCt^{-1}=C}$ using the identity map, so on starting with an embedding
of a free group and applying Theorem 6.1 repeatedly then dropping to a
subgroup, we have that any limit group embeds in $\slt$. This result was
also obtained in \cite{fr} using similar techniques but they only
considered the case of hyperbolic limit groups.  

The HNN extensions permitted in Theorem 6.1 can only identify two conjugate
maximal abelian subgroups and can only use a conjugation map as the 
identifying isomorphism, which is very restrictive. However the next
Proposition, which is really just a restatement of the point mentioned in
Section 2 that the Klein bottle group does not embed in $\slt$, shows
that such restrictive conditions are required.
\begin{prop}
Let $G$ embed in $\slt$ and let $\Gamma$ be an HNN extension of $G$ with
stable letter $t$ which
identifies two isomorphic subgroups $A,B$ of $G$ using
an isomorphism $\theta$. If there is $a\in A$ (where $a\neq\pm I$)
and $b\in B$ with
$\theta(a)=b$ such that $b$ is conjugate in $G$ to
$a^{-1}$ then $\Gamma$ does not embed in $\slt$.
\end{prop}
\begin{proof} We have $tat^{-1}=ga^{-1}g^{-1}$ holding in $\Gamma$ and if
this held in $\slt$ then the matrix $g^{-1}t$ could only have order 4,
but in any HNN extension an element which is equal to the identity must have 
zero exponent sum in the stable letter. 
\end{proof}
In particular, although we mentioned in Section 3 that a graph of non
abelian free groups with maximal cyclic edge groups embeds in $\slt$ if
the graph is a tree, we see that this need not be true for graphs with
loops.

All of the HNN extensions in Proposition 6.2 which fail to embed in
$\slt$ contain $\Z\times\Z=\langle a^2,g^{-1}t\rangle$, at least if $a$ is of
infinite order, and so will not be word hyperbolic. We 
now give an example of a graph of non
abelian free groups with maximal cyclic edge groups which is word hyperbolic
but which does not embed in $\slt$.
We have a result in \cite{bfcr} giving conditions on when an HNN extension
over a virtually cyclic group is word hyperbolic, alongside the theorem
already mentioned in \cite{bf} for amalgamated free products. For a torsion
free word hyperbolic group $G$, this states that if $A,B$ are infinite
cyclic subgroups of $G$ then  
an HNN extension formed by identifying $A$ and $B$ is word hyperbolic
if and only if for all $g\in G$ we have
$gAg^{-1}\cap B=\{e\}$ and where at least one of $A$ and $B$ is generated
by a primitive element. This is also equivalent to saying that the HNN
extension does not contain a Baumslag Solitar subgroup.
(This result is false if the primitive condition is
removed, as was originally stated in \cite{bf}, hence necessitating the
appearance of \cite{bfcr}.) Once again we cannot use this result
repeatedly as it stands but if $A$ and $B$ are both generated by primitive
elements then we can take multiple HNN extensions and still obtain word
hyperbolicity by this result and Lemma 2.6. We do this
until too many elements are forced to be conjugate to each 
other.    
\begin{prop}
Let $F_2$ be free on $a,b$ and let $\Gamma$ be the triple HNN extension
formed using stable letters $t,s,r$ so that $tat^{-1}=b,sas^{-1}=ab$
and $rar^{-1}=aba^{-1}b^{-1}$. Then $\Gamma$ does not embed in $\slt$
but is a graph of non
abelian free groups with maximal cyclic edge groups which is word hyperbolic.
\end{prop}
\begin{proof}
If $\Gamma$ did embed in $\slt$ then $a$ and $b$ would generate a rank two
free group where $a,b,ab$ and $aba^{-1}b^{-1}$ all have the same trace,
$z\in\C$ say.
But using the well known trace identities in $\slt$
which go back to Fricke and
Klein, we have that
\[ \mbox{tr}^2(a) +\mbox{tr}^2(a)+\mbox{tr}^2(ab)-2
= \mbox{tr}(a)\; \mbox{tr}(b)\; \mbox{tr}(ab) \; +\mbox{tr}(aba^{-1}b^{-1})\]
so $(z-2)(z^2-z-1)=0$, but $z=2$ would imply that $\langle a,b\rangle$ is
metabelian and we have elements of order 5 otherwise.

Now $\Gamma$ is formed using a graph of groups consisting of one vertex
representing $F_2$ and three loops for the three pairs of edge groups, all
of which are maximal cyclic in $F_2$. To see that $\Gamma$ is word 
hyperbolic, we apply Bestvina and Feighn's result above each time in 
conjunction with Lemma 2.6, noting that $a,b,ab,aba^{-1}b^{-1}$ are non
conjugate primitive elements of $F_2$.
\end{proof}

We finish this section by mentioning some recent work by Hsu and Wise
on the linearity of graphs of free groups with infinite cyclic edge
groups. In \cite{hsw} it was shown that if such a group is word
hyperbolic then it is linear over $\Z$, although one would expect the
dimension to be very high. Now
the Bestvina-Feighn results tell us that such a group is
word hyperbolic if and only if it does not contain any Baumslag Solitar
subgroup.
Let us call a Baumslag Solitar group $BS(m,n)$ Euclidean if
$|n|=|m|$ and non Euclidean otherwise. 
Then a non Euclidean Baumslag Solitar group cannot be
linear over $\Z$ (as either it fails to be residually finite or it
is soluble but not polycyclic). However the Euclidean ones are linear over
$\Z$ (as they have a finite index subgroup $F_n\times\Z$) so it is possible
that the necessary condition of not containing a non Euclidean 
Baumslag Solitar subgroup
is sufficient to make a group $G$ which splits as a
graph of free groups with infinite cyclic edge
groups be linear over $\Z$. 
Although this is currently open, the paper does show that $G$ is the 
fundamental group of a compact non positively curved cube complex 
if this condition holds.

\section{Minsky's simple loop question}
Suppose that $S$ and $T$ are both closed orientable surfaces of genus 
at least two. Theorem 2.1 in \cite{gab} showed that if $f:S\rightarrow T$ is
a continuous map such that the induced homomorphism $f_*:\pi_1(S)
\rightarrow\pi_1(T)$ is not injective then there exists a non
contractible simple closed curve $\alpha$ in $S$ such that $f(\alpha)$ is
a homotopically trivial closed curve in $T$. Note that any homomorphism
$\theta:\pi_1(S)\rightarrow\pi_1(T)$ is induced by a continuous map from
$S$ to $T$ because both surfaces are aspherical. This demonstrates that
the simple closed curves on a surface form a special subset of all closed
curves as the kernel of any non injective homomorphism to the fundamental
group of another surface must meet this subset. We can try to generalise
this result by replacing the codomain $\pi_1(T)$ with other groups.
Indeed Problem 3.96 in \cite{kir} which is
entitled the simple loop conjecture for
3-manifolds asks: Let $f:S\rightarrow M^3$ be a 2 sided immersion of a
closed surface $S$ in a 3-manifold $M^3$ 
such that $f_*$ is not injective then does there exist an
essential simple loop on $S$ whose image is null homotopic in $M$?
The answer is known to be yes
for Siefert fibred spaces but the question is still open for
$M^3$ an orientable hyperbolic 3-manifold.
In this case $\pi_1(M^3)$ will be a discrete subgroup of $\slt$ so a variation
is the following problem by Minsky. Question 5.3 of \cite{min}
asks whether all non injective homomorphisms
from $\pi_1(S)$ to $\slt$ have an element in the kernel which is
represented by a non contractible simple closed curve on $S$.
Of course if this were true then it would apply both when the image is
discrete and when indiscrete, thus solving the above open question.
However we have seen in this paper many examples of subgroups of $\slt$
which are not isomorphic to discrete groups, so it should not be too
surprising if counterexamples did exist to Minsky's question.
(In fact Minsky asked this question for $P\slt$ but a homomorphism to
$\slt$ with image missing $-I$ will work here too.)

Indeed there are counterexamples to this question, as shown in \cite{cm}
using character varieties, in \cite{ll} by obtaining homomorphisms from
surface groups whose images are limit groups and thus embed in $\slt$, and
a short proof in \cite{clsl} using stable commutator length. In fact our
approach has aspects in common with the last two papers, in that we also
use a theorem of Hempel from 1990 which is quoted in \cite{ll} and, as in
\cite{clsl}, we are considering cyclic amalgamations. However we need
only apply an argument in Section 3 about when these groups embed in $\slt$, 
along with this result of Hempel and the work of Whittemore from 1973 (on
the $\slt$ representations of the fundamental group of the figure eight knot 
complement mentioned in Section 5) in order to obtain a very quick and
explicit proof. 
\begin{thm}
Let $\pi_1(S_2)=\langle A,B,C,D|ABA^{-1}B^{-1}=CDC^{-1}D^{-1}\rangle$
be the fundamental group of the closed orientable surface of genus two
and let $N$ be the normal closure of the element
$r(A,B)=B^{-1}A^{-1}BAB^{-1}ABA^{-1}B^{-1}A$ in $\pi_1(S_2)$. Then
$N$ is non trivial but does not contain any element which is
represented by a simple closed curve on $S_2$. However $\pi_1(S_2)/N$ embeds
in $\slt$ and misses $-I$.
\end{thm}
\begin{proof}
The group $\pi_1(S_2)$ is of course formed by amalgamating the free group
on $A,B$ with the free group on $C,D$ over $ABA^{-1}B^{-1}=CDC^{-1}D^{-1}$
and $r(A,B)$ is clearly non trivial in the first free group which injects
into $\pi_1(S_2)$. However by drawing the curve $r(A,B)$ on $S_2$ 
using the standard identification of the fundamental domain we see
that it is certainly not simple. Now a result in \cite{hemscc} states that
if $S$ is a closed orientable surface and $\alpha\in\pi_1(S)$ is a closed
curve such that the normal closure of $\alpha$ in $\pi_1(S)$ contains a
non trivial simple closed curve $\beta$ then $\alpha$ is itself a simple 
closed curve and $\beta^{\pm 1}$ is homotopic either to $\alpha$ or to
the commutator of $\alpha$ with some simple closed curve $\gamma$ meeting
$\alpha$ transversely in a single point. An immediate consequence of
this in our case is that as the closed curve $r(A,B)\in\pi_1(S_2)$ is not
simple, $N$ contains no non trivial simple closed curves.

We just need to show that $\pi_1(S_2)/N$ embeds in $\slt$. But rather
than think of this as taking the free group on $A,B,C,D$ and adding first
the surface relation, then the relation given by $r(A,B)=I$, we can do this
the other way around. Of course the word $r(A,B)$ has been seen before in
Section 5 where $G=\langle A,B|r(A,B)\rangle$ was the fundamental group
of the figure eight knot complement. Hence $\pi_1(S_2)/N$ is also an
amalgamation of this group $G$ and the free group $F$ on $C,D$
formed by identifying $ABA^{-1}B^{-1}$ in $G$ with $CDC^{-1}D^{-1}$ in $F$.
We can now apply Corollary 3.5 to show that $\pi_1(S_2)$ embeds in $\slt$:
it is clear
that $CDC^{-1}D^{-1}$ is a primitive element of $F$ and
we can embed a rank 2 free group in $\slt$ so that every non trivial
element has a transcendental trace. As for $ABA^{-1}B^{-1}$ in $G$,
this is the element $ab$ when written in terms of the generators $t,a,b$
giving the fibred form of the figure eight knot. Since the fibre subgroup
is free on $a,b$ we have that $ABA^{-1}B^{-1}$ lies in this subgroup and
is a primitive element of it. Moreover the centraliser $C_G(ABA^{-1}B^{-1})$
in $G$ is also infinite cyclic, because $ABA^{-1}B^{-1}$ has trace
$z-1$ which is not equal to $\pm 2$ in the discrete faithful representation 
of $G$ in $\slt$ where $x=2$ and $z=2-\omega$. 
Thus $ABA^{-1}B^{-1}$ must also be primitive in $G$ 
because any element of $G$ outside the fibre subgroup has all its non
trivial powers outside of this fibre subgroup. 
Thus on taking $\lambda$ to be
any transcendental number as in Proposition 5.3, we have an embedding
of $G$ in $\slt$ in which the primitive
element $ABA^{-1}B^{-1}$ has transcendental trace, so the
conditions of Corollary 3.5 are satisfied.

Finally as $\pi_1(S_2)/N$ is an amalgamated free product of two torsion
free groups, it is also torsion free and so misses $-I$ in $\slt$, implying
that it embeds in $P\slt$ too. 
\end{proof}

\end{document}